\documentclass[12pt,a4paper]{article}
\usepackage{amsmath,amssymb,graphics,mathptmx,fullpage}
\newcommand{\fs}{\{0,1\}^*}
\newcommand{\cs}{\{0,1\}^\infty}
\newcommand{\KA}{\textit{KA}\,}
\newcommand{\KS}{C}
\newcommand{\KP}{K}

\DeclareMathOperator{\Int}{\textrm{Int}}
\let\ge=\geqslant
\let\le=\leqslant
\newtheorem{theorem}{Theorem}
\newenvironment{proof}{\textbf{Proof}.}{\raisebox{-3pt}{\hbox{$\Box$}}}

\begin{document}
\title{Limit complexities revisited [once more]}
\date{}
\author{Laurent Bienvenu$^1$, Andrej Muchnik$^2$,\\
Alexander Shen$^{1,4,5}$ and Nikolay Vereshchagin$^{3,4}$}

\maketitle

\footnotetext[1]{Laboratoire d'Informatique Fondamentale,
    CNRS \& Universit\'e Aix-Marseille, France. Supported in part by
    ANR Sycomore and NAFIT ANR-08-EMER-008-01 grants.}

\footnotetext[2]{%
  Andrej Muchnik (24.02.1958 -- 18.03.2007)
  worked in the Institute of New Technologies in
  Education (Moscow). For many years he participated in
  Kolmogorov seminar at the Moscow State (Lomonosov) University.
  N.~Vereshchagin and A.~Shen (also participants of that
  seminar) had the privilege to know Andrej for more than two
  decades and are deeply indebted to him both as a great thinker
  and noble personality. The text of this paper was written
  after Andrej's untimely death but it (like many other papers
  written by the participants of the seminar) develops his
  ideas.}

\footnotetext[3]{%
Moscow State Lomonosov University, Russia. Supported in part by
RFBR 05-01-02803-CNRS-a, 06-01-00122-a.}

\footnotetext[4]{%
CNRS Poncelet Laboratory, Moscow}

\footnotetext[5]{%
IITP RAS, Moscow
}

\begin{abstract}
The main goal of this article is to put some known results in a
common perspective and to simplify their proofs.

We start with a simple proof of a result of
Vereshchagin~\cite{ver} saying that $\limsup_{n}\KS(x|n)$ (here
$\KS(x|n)$ is conditional (plain) Kolmogorov complexity of $x$
when $n$ is known) equals $\KS^{\mathbf{0}'}(x)$, the plain
Kolmogorov complexity with $\mathbf{0}'$-oracle.

Then we use the same argument to prove similar results for
prefix complexity, a priori probability on binary tree, to prove Conidis'
theorem~\cite{conidis} about limits of effectively open sets, and also to improve the results of Muchnik~\cite{muchnik} about limit frequencies. As a by-product,
we get a criterion of $\mathbf{0}'$ Martin-L\"of randomness
(called also $2$-randomness) proved in Miller~\cite{miller}: a
sequence $\omega$ is $2$-random if and only if there exists $c$
such that any prefix $x$ of $\omega$ is a prefix of some string
$y$ such that $\KS(y)\ge |y|-c$. (In the 1960ies this property
was suggested in Kolmogorov~\cite{kolmogorov} as one of possible
randomness definitions; its equivalence to $2$-randomness was
shown in Miller~\cite{miller}). Miller~\cite{miller} and Nies et
al.~\cite{nies} proved another $2$-randomness criterion:
$\omega$ is $2$-random if and only if $\KS(x)\ge |x|-c$ for some
$c$ and infinitely many prefixes $x$ of~$\omega$. This criterion is also
a consequence of the results mentioned above.

[The original version of this work~\cite{stacs-tocs} contained a weaker (and cumbersome) version of
Conidis' result, and the proof used low basis theorem (in quite a strange way). The full version was formulated as a conjecture. This conjecture was later proved by Conidis. Bruno Bauwens (personal communication) noted that the proof can be obtained also by a simple modification of our original argument, and we reproduce Bauwens' argument with his permission.]
\end{abstract}

%=========================================
\section{Plain complexity}
\label{sec:plain-complexity}
%=========================================

We denote by $\fs$ the set of binary strings and by $\cs$ the
set of infinite binary sequences. For $x \in \fs$, we denote by
$\KS(x)$ the plain complexity of~$x$ (the length of the shortest
description of~$x$ when an optimal description method is fixed,
see Li and Vitanyi~\cite{li-vitanyi}; no requirements about
prefixes). By $\KS(x|n)$ we mean conditional complexity of~$x$
when~$n$ is given, see for example Li and
Vitanyi~\cite{li-vitanyi}. Superscript $\mathbf{0}'$ in
$\KS^{\mathbf{0}'}$ means that we consider the relativized
version of complexity to the oracle $\mathbf{0}'$, the universal
computably enumerable set.

The following result was proved in Vereshchagin~\cite{ver}. We
provide a simple proof for it.

\begin{theorem}
        \label{plain}
For all~$x \in \fs$\textup{:}
        $$
\limsup_{n\to\infty}\, \KS(x|n)=\KS^{\mathbf{0}'}(x)+O(1).
        $$
\end{theorem}

\noindent
(In this theorem and below ``$f(x)=g(x)+O(1)$'' means that there
is a constant~$c$ such that $|f(x)-g(x)|\le c$ for all $x$.)

\begin{proof}
We start in the easy direction. Let $\mathbf{0}_n$ be the
(finite) set consisting of the elements of the universal
enumerable set $\mathbf{0}'$ that have been enumerated after~$n$
steps of computation (note that $\mathbf{0}_n$ can be computed
from~$n$). If $\KS^{\mathbf{0}'}(x)\le k$, then there exists a
description (program) of size at most $k$ that generates $x$
using $\mathbf{0}'$ as an oracle. Only finite part of the oracle
can be used in the computation that produces $x$, so $\mathbf{0}'$ can be replaced by $\mathbf{0}_n$
for all sufficiently large $n$, and oracle $\mathbf{0}_n$ can be
reconstructed if $n$ is given as a condition. Therefore,
$\KS(x|n)\le k+O(1)$ for all sufficiently large $n$, and
        $$
\limsup_{n\to\infty}\,\KS(x|n)\le \KS^{\mathbf{0}'}(x)+O(1).
        $$

For the reverse inequality, fix $k$ and assume that
$\limsup\,\KS(x|n) <k$. This means that for all sufficiently
large $n$ the string $x$ belongs to the set
        $$
U_n=\{ u \mid \KS(u|n)<k\}.
        $$
The family $U_n$ is an enumerable family of sets (given $n$ and
$k$, we can generate $U_n$); each of these sets has at most
$2^k$ elements. We need to construct a $\mathbf{0}'$-computable
process that given $k$ generates at most $2^k$ elements
including all elements that belong to $U_n$ for all sufficiently
large~$n$. (Then strings of length $k$ may be assigned as
$\mathbf{0}'$-computable codes of all generated elements.)

To describe this process, consider the following operation: for
some $u$ and $N$ add $u$ to all $U_n$ such that $n\ge N$. (In
other terms, we add a horizontal ray starting from $(N,u)$ to
the set $\mathcal{U}=\{(n,u)\mid u\in U_n\}$.) This operation is
\emph{acceptable} if all $U_n$ still have at most $2^k$ elements
after it (i.e., if before this operation all $U_n$ such that
$n\ge N$ either contain $u$ or have strictly less than $2^k$ elements).

For any given triple $u$, $N$, $k$, we can find out using
$\mathbf{0}'$-oracle whether this operation is acceptable or
not. Indeed, the operation is not acceptable if and only if some
$U_n$ for $n\ge N$ contains at least $2^k$ elements that are
distinct from~$u$. Formally, the operation is not acceptable if
    $$
(\exists n \ge N)\, \left| U_n \setminus \{u\}\right| \ge 2^k,
    $$
and this is an enumerable condition as the $U_n$ are themselves
enumerable. Now for all pairs $(N,u)$ (in some computable order)
we perform the $(N,u)$-operation if it is acceptable. (The
elements added to some $U_i$ remain there and are taken into
account when next operations are attempted.) This process is
$\mathbf{0}'$-computable since after any finite number of
operations the set $\mathcal{U}$ is enumerable (without any
oracle) and its enumeration algorithm can be
$\mathbf{0}'$-effectively found (uniformly in~$k$).

Therefore the set of all elements $u$ that participate in
acceptable operations during this process is uniformly
$\mathbf{0}'$-enumerable. This set contains at most $2^k$
elements (otherwise $U_n$ would become too big for large $n$).
Finally, this set contains all $u$ such that $u$ belongs to the
(original) $U_n$ for all sufficiently large $n$. Indeed, the
operation is always acceptable if the element we want to add is
already present!
\end{proof}

\medskip

The proof has the following structure. We have an enumerable
family of sets $U_n$ that all have at most $2^k$ elements. This
implies that the set
        $$
U_\infty=\liminf_{n\to\infty} U_n
        $$
has at most $2^k$ elements where, as usual, the $\liminf$ of a
sequence of sets is the set of elements that belong to almost
all sets of the sequence. If~$U_\infty$ were
$\mathbf{0}'$-enu\-me\-ra\-ble, we would be done. However, this
may be not the case: the criterion
        $$
u\in U_{\infty} \Leftrightarrow \exists N\,(\forall n\ge N)\, [u\in U_n]
        $$
has $\exists\forall$ prefix before an enumerable (not
necessarily decidable) relation, that is, one quantifier more
than we want (to guarantee that $U_\infty$ is
$\mathbf{0}'$-enumerable). However, in our proof we managed to
cover $U_\infty$ by a set that is $\mathbf{0}'$-enumerable and
still has at most $2^k$ elements.

%============================================================
\section{Prefix complexity and a priori probability}
\label{sec:prefix-complexity}
%============================================================

We now prove a similar result for prefix complexity (or, in
other terms, for a priori probability). Let us recall the
definition. The function $a(x)$ on binary strings (or integers)
with non-negative real values is called a \emph{semimeasure} if
$\sum_x a(x) \le 1$. The function~$a$ is \emph{lower
semicomputable} if there exists a computable total function
$(x,n)\mapsto a(x,n)$ with rational values such that for
every~$x$ the sequence $a(x,0), a(x,1),\ldots$ is a
nondecreasing sequence that has limit~$a(x)$.

There exists a maximal (up to a constant factor) lower
semicomputable semimeasure $m$ (see, e.g., Li and
Vitanyi~\cite{li-vitanyi}). The value $m(x)$ is sometimes called
the \emph{a priori probability} of~$x$. In the same way we can
define \emph{conditional} a priory probability $m(x|n)$ and
$\mathbf{0}'$-\emph{relativized} a~priori probability
$m^{\mathbf{0}'}(x)$ (which is a maximal semimeasure among the
$\mathbf{0}'$-lower semicomputable ones\label{def:moprime}).

\begin{theorem}
        \label{apriori}

For all~$x \in \fs$\textup{:}
        $$
\liminf_{n\to\infty}\,m(x|n)=m^{\mathbf{0}'}(x)
        $$
up to a $\Theta(1)$ multiplicative factor \textup{(}in other
terms, two inequalities with $O(1)$ factors hold\textup{)}.
\end{theorem}

\begin{proof}
If $m^{\mathbf{0}'}(x)$ is greater than some $\varepsilon$, then
for sufficiently large $n$ the value $m^{\mathbf{0}_n}(x)$ is
also greater than $\varepsilon$. (Indeed, this inequality is
established at some finite stage when only a finite part of
$\mathbf{0}'$ is used.) We may assume without loss of generality
that the function $x\mapsto m^A(x)$ is a semimeasure for any $A$
(recalling the construction of the maximal semimeasure). Then,
similarly to the previous theorem, we have
        $$
\liminf_{n\to\infty}\,m(x|n)\ge
\liminf_{n\to\infty}\,m^{\mathbf{0}_n}(x)\ge
m^{\mathbf{0}'}(x)
        $$
up to constant multiplicative factors. Indeed, for the first
inequality, notice that we can define a conditional lower
semicomputable semimeasure $\mu$ by $\mu(x | n) =
m^{\mathbf{0}_n}(x)$. By maximality of~$m$, we have $\mu(x | n)
\leq m(x | n)$ for all $x,n$, up to a multiplicative factor. For
the second inequality, recall that $m^{\mathbf{0}'}(x)$ is the
nondecreasing limit of an $\mathbf{0}'$-computable sequence
$m^{\mathbf{0}'}(x,0), m^{\mathbf{0}'}(x,1), \ldots$. Let $s$ be
such that $m^{\mathbf{0}'}(x,s) \geq \frac{1}{2}
m^{\mathbf{0}'}(x)$. Since the computation of
$m^{\mathbf{0}'}(x,s)$ only uses finitely many bits of
$\mathbf{0}'$, we have for all large enough~$n$:
$m^{\mathbf{0}_n}(x,s)=m^{\mathbf{0}'}(x,s) \geq \frac{1}{2}
m^{\mathbf{0}'}(x)$ and thus $m^{\mathbf{0}_n}(x) \geq
\frac{1}{2} m^{\mathbf{0}'}(x)$.\\

The other direction of the proof is also similar to the second
part of the proof of Theorem~\ref{plain}. Instead of enumerable
finite sets $U_n$ we now have a sequence of (uniformly) lower
semicomputable functions $x\mapsto m_n(x)=m(x|n)$. Each of the
$m_n$ is a semimeasure. We need to construct an
$\mathbf{0}'$-lower semicomputable semimeasure $m'$ such that
        $$
m'(x)\ge \liminf_{n\to\infty}\,m_n(x)
        $$
Again, the $\liminf$ itself cannot be used as $m'$: we do have
${\sum_x \liminf_n m_n(x)\le 1}$ as $\sum_x m_n(x)\le 1$ for all
$n$, but unfortunately the equivalence
        $$
r < \liminf_{n\to\infty} m_n(x) \Leftrightarrow
(\exists r'>r)(\exists N)\, (\forall n\ge N)\, [r'<m_n(x)]
        $$
has too many quantifier alternations (one more than needed; note
that the quantity $m_n(x)$ is lower semicomputable making the
$[\ldots]$ condition enumerable). The similar trick helps. For a
triple $(r,N,u)$ consider an \emph{increase operation} that
increases all values $m_n(u)$ such that $n\ge N$ up to a given
rational number $r$ (not changing them if they were greater than
or equal to~$r$). This operation is \emph{acceptable} if all
$m_n$ remain semimeasures after the increase.

The question whether the increase operation is acceptable is
$\mathbf{0}'$-decidable. And if it is acceptable, by performing
it we get a new (uniformly) lower semicomputable sequence of
semimeasures. We can then try to perform an increase operation
for some other triple. Doing that for all triples (in some
computable ordering), we can then define $m'(u)$ as the upper
bound of $r$ for all successful $(r,N,u)$ increase operations
(for all $N$). This gives a $\mathbf{0}'$-lower semicomputable
function; it is a semimeasure since we verify the semimeasure
inequality for every successful increase attempt; finally,
$m'(u) \ge \liminf\, m_n(u)$ since if $m_n(u)\ge r$ for all
$n\ge N$, then the $(r,N,u)$-increase does not change anything
and is guaranteed to be acceptable at any step.~%
\end{proof}

\bigskip
The expression $-\log m(x)$, where $m$ is the maximal lower
semicomputable semimeasure, equals the so-called \emph{prefix}
complexity $\KP(x)$ (up to an additive $O(1)$ term; see for
example Li and Vitanyi~\cite{li-vitanyi}). The same is true for
relativized and conditional versions, and we get the following
reformulation of the last theorem:

\begin{theorem}
        \label{prefix}
        $$
\limsup_{n\to\infty}\,\KP(x|n)= \KP^{\mathbf{0}'}(x)+O(1).
        $$
\end{theorem}

\bigskip

Another corollary improves a result of Muchnik~\cite{muchnik}.
For any (partial) function $f$ from $\mathbb{N}$ to $\mathbb{N}$
let us define the \emph{limit frequency} $q_f(x)$ of an integer
$x$ as
        $$
q_f(x)=\liminf_{n\to\infty}\, \frac{\#\{i<n\mid f(i)=x\}}{n}
        $$
In other words, we look at the fraction of values~$x$
among the first $n$ values $f(0),\ldots,{f(n-1)}$ of~$f$
(undefined values are also listed) and take the $\liminf$ of
these fractions. It is easy to see that for a total computable
$f$ the function $q_f$ is a lower $\mathbf{0}'$-semicomputable
semimeasure. Moreover, it is shown in Muchnik~\cite{muchnik}
that any $\mathbf{0}'$-semicomputable semimeasure $\mu$ can be
represented as $\mu=q_f$ for some computable function~$f$. In
particular this implies that there exists a total computable
function $f$ such that $q_f=m^{\mathbf{0}'}$.\\

We would like to extend Muchnik's result to partial computable
functions~$f$. The problem is that if~$f$ is only partial
computable, the function~$q_f$ is no longer guaranteed to be
lower semicomputable. Using the second part of the proof of
Theorem~\ref{apriori}, we can nonetheless prove:

\begin{theorem}
        \label{frequency}
For any partial computable function~$f$, the function $q_f$ is
upper bounded by a lower $\mathbf{0}'$-semicomputable
semimeasure.
\end{theorem}

\begin{proof}
Indeed, given a partial computable function~$f$, we can define
for all~$n$ a semimeasure~$\mu_n$ as
     $$
\mu_n(x)=\frac{\#\{i<n\mid f(i)=x\}}{n};
     $$
$\mu_n$ is lower semicomputable uniformly in $n$. Then
$q_f=\liminf \mu_n$; on the other hand we know from the proof of
Theorem~\ref{apriori} that the $\liminf$ of a sequence of
(uniformly) lower semicomputable semimeasures is bounded by a
$\mathbf{0}'$-lower semicomputable semimeasure. The result
follows.~%
\end{proof}
\bigskip

The same type of argument also is applicable to the so-called
\emph{a priori complexity} defined as negative logarithm of a
maximal lower semicomputable semimeasure on the binary tree (see
Zvonkin and Levin~\cite{zvonkin-levin}). This complexity is
sometimes denoted as $\KA(x)$ and we get the following
statement:

\begin{theorem}
        \label{aprioritree}
        $$
\limsup_{n\to\infty}\KA(x|n)=\KA^{\mathbf{0}'}(x)+O(1).
        $$
\end{theorem}

(To prove this we define an increase operation in such a way
that, for a given lower semicomputable semimeasure on the binary
tree~$a$, it increases not only $a(x)$ but also $a(y)$ for $y$
that are prefixes of $x$, if necessary. The increase is
acceptable if $a(\Lambda)$ still does not exceed~$1$.)

\medskip

It would be interesting to find out whether similar results are
true for monotone complexity or not (the authors do not know
this).

%===================================================
\section{Open sets of small measure}
\label{sec:small-open-sets}
%===================================================

In Section~\ref{sec:plain-complexity} we covered the $\liminf$
of a sequence of finite uniformly enumerable sets $U_i$ by a
$\mathbf{0}'$-enumerable set $V$ that is essentially no bigger
than the $U_i$. It was done in a uniform way, i.e.,~$V$ can be
effectively constructed given the enumerations of the $U_i$ and
an upper bound for their cardinalities. We now look at the
continuous version of this problem where the $U_i$ are open sets
of small measure.

We consider open sets in the Cantor space $\cs$ (the set of all
infinite sequences of zeros and ones). An \emph{interval} $[x]$
(for a binary string $x$) is formed by all sequences that have
prefix $x$. Open sets are unions of intervals. An
\emph{effectively open} subset of $\cs$ is an enumerable union
of intervals, i.e., the union of intervals $[x]$ where strings
$x$ are taken from some enumerable set.

We consider standard (uniform Bernoulli) measure on $\cs$: the
interval $[x]$ has measure $2^{-l}$ where~$l$ is the length
of~$x$.

A classical theorem of measure theory says: 
\begin{quote}\emph{if
$U_0,U_1,U_2,\ldots$ are open sets of measure at most
$\varepsilon$, then $\liminf_n U_n$ has measure at most
$\varepsilon$}, and this implies that \emph{for every
$\varepsilon'>\varepsilon$ there exists an open set of measure
at most $\varepsilon'$ that covers $\liminf_n U_n$.}
\end{quote}
Indeed,
        $$
\liminf_{n\to\infty}\,U_n = \bigcup_{N} \bigcap_{n\ge N} U_n,
        $$
and the measure of the union of an increasing sequence
        $$
V_N = \bigcap_{n\ge N} U_n,
        $$
equals the limit of measures of $V_N$, and all these measures do not exceed $\varepsilon$ since $V_N\subset U_N$. Recall also that for any measurable subset $X$ of $\cs$ its measure $\mu(X)$ is the infimum of the measures of open sets that cover~$X$.

We now can ``effectivize'' this statement in the same way
as we did before. In Section~\ref{sec:plain-complexity} we
started with an (evident) statement: \emph{if $U_n$ are finite
sets of at most $2^k$ elements, then $\liminf_n U_n$ has at most
$2^k$ elements} and proved its effective (in the halting
problem) version: \emph{for a uniformly enumerable family of
finite sets $U_n$ that have at most $2^k$ elements, the set
$\liminf_n U_n$ is contained in a uniformly
$\mathbf{0}'$-enumerable set that has at most $2^k$ elements.}

In Section~\ref{sec:prefix-complexity} we did a similar thing
with semimeasures. Again, the non-effective version is trivial:
it says that if $\sum_x m_n(x)\le 1$ for every~$n$, then
$\sum_x\liminf_n m_n(x)\le 1$. We have proved the effective
version that provides a $\mathbf{0}'$-semicomputable semimeasure
that is an upper bound for $\liminf m_n$.

For the statement about $\liminf U_n$, the effective
version is the following statement, proved in full generality by Conidis~\cite{conidis}. (In the previous version of this paper only a much weaker and more obscure statement was proven, and the full version was formulated as a conjecture.) 

\begin{theorem}[Conidis]\label{conidis}
Let $\varepsilon>0$ be a
rational number and let $U_0,U_1,\ldots$ be an enumerable family
of effectively open sets of measure at most $\varepsilon$ each.
Then for every rational $\varepsilon'>\varepsilon$ there exists
a $\mathbf{0}'$-effectively open set $V$ of measure at most
$\varepsilon'$ that contains $\liminf_{n\to\infty} U_n=
\bigcup_{N} \bigcap_{n\ge N} U_n$, and the $\mathbf{0}'$-enumeration algorithm for $V$ can be effectively found given $\varepsilon$, $\varepsilon'$, and the enumeration algorithm for $U_i$.
\end{theorem}

\begin{proof}
Let us first try the same trick as above. For every interval $[x]$ and for every 
natural $i$ we may try to add $[x]$ to all $U_i,U_{i+1},\ldots$ and see whether the restriction on the measure of $U_n$ is now violated (i.e., some of the enlarged $U_n$ have now measure greater than $\varepsilon$). This can be effectively tested with the help of $\textbf{0}'$-oracle. If the restriction is violated, this pair $(x,i)$ is ignored; if the restriction is still satisfied, we add $[x]$ to all $U_i,U_{i+1},\ldots$ and use the enlarged sets in the sequel.

The process is $\textbf{0}'$-computable, and the union of all added intervals is an $\textbf{0}'$-effectively open set of measure at most $\varepsilon$. However, trying to prove that this open set covers $\liminf U_n$ (i.e., covers $V_N$ for all $N$, see above), we encounter a problem. We can be sure that some pair $(x,N)$ is accepted (and the interval $[x]$ is added starting from $N$th position) if $[x]$ already belongs to $U_N,U_{N+1},\ldots$; in this case $[x]$ is a subset of $\Int (V_N)$.  (By $\Int X$ we mean a maximal open subset of $X$.) So this reasoning gives only a weaker statement (proved in~\cite{stacs-tocs}): the set
     $$
 \bigcup_N \Int \big( \bigcap_{n\ge N} U_n\big)   
     $$
can be covered by a $\mathbf{0}'$-effectively open set of small measure.

To get a desired statement, we need do modify the procedure. This modification was suggested by Bruno Bauwens~\cite{bauwens}. (The original proof of Conidis is  indirect: he first covers the required set up to a null set.)

First, we need some tolerance to the measure increase when we attempt to add some interval $[x]$ starting from the set number $i$: the threshold (initially $\varepsilon$) increases at this step by some $\delta_{x,i}$. The computable family of rational numbers $\delta_{x,i}>0$ is selected in such a way that the sum of all $\delta_{x,i}$ does not exceed $\varepsilon'-\varepsilon$.

Second, after we see that the attempt (to add $[x]$ to $U_i,U_{i+1},\ldots$) is unsuccessful because the (increased) threshold is crossed, we do not give up. Instead, we select a first $m$ for which $U_m$ becomes too big after adding $[x]$, and replace $[x]$ by $[x]\cap U_m$: we then try to add $[x]\cap U_m$ to $U_i,U_{i+1},\ldots$ instead  of $[x]$. May be again the attempt is unsuccessful and some $U_{t}$ (for some $t>m$) again crosses the same threshold.  Then we take the intersection $[x]\cap U_m\cap U_t$ and so on. Note that each new intersection operation decreases the size of the added set by $\delta_{x,i}$, since the outstanding part, now eliminated, was at least of this size. So this process of ``trimming'' is finite and at some point we add the trimmed set $[x]\cap U_m\cap U_t\cap\ldots\cap U_v$ without exceeding the threshold.

It remains to show that in this way we indeed cover $\liminf U_n$. Indeed, assume that some sequence $\alpha$ belongs to all $U_N,U_{N+1},\ldots$. Then, starting to add some interval containing $\alpha$ to $U_N,U_{N+1},\ldots$, we will never remove $\alpha$ by trimming, so $\alpha$ will be covered.
\end{proof}

\textbf{Remark}. In fact the intervals $[x]$ are not needed in this argument, we can start every time from the entire Cantor space. Then the proof can be reformulated as follows. Let us denote by $U_{k..l}$ the intersection $U_{k}\cap U_{k+1}\cap\ldots\cap U_l$. Fix an increasing computable sequence $\varepsilon<\varepsilon_1<\varepsilon_2<\ldots<\varepsilon'$. There exists some $k_1$ such that for every $i>k_1$ the set
      $$
U_{1..k_1} \cup U_i      
      $$
 has measure at most $\varepsilon_1$. (Indeed, if for some $i$ the measure is greater than $\varepsilon_1$, then, adding $U_i$ as a new term in the intersection, we decrease the measure of the intersection at least by $\varepsilon_1-\varepsilon$; such a decrease may happen only finitely many times.) For similar reasons we can then find $k_2$ such that for every $i$ the set
      $$
U_{1..k_1} \cup U_{k_1+1..k_2} \cup U_i  
      $$
has measure at most $\varepsilon_2$ for every $i>k_2$. And so on. This construction is $\mathbf{0}'$-computable and the union
     $$
U_{1..k_1} \cup U_{k_1+1..k_2} \cup U_{k_2+1..k_3} \cup\ldots 
      $$
is an $\mathbf{0}'$-effectively open cover of $\liminf U_n$ of measure at most $\varepsilon'$.

%=====================================================
\section{Kolmogorov and $2$-randomness}
\label{sec:kolmo-2-rand}
%=====================================================

Theorem~\ref{conidis} has an historically remarkable
corollary. When Kolmogorov tried to define randomness in
1960s, he started with the following approach. A string $x$ of
length $n$ is ``random'' if its complexity $\KS(x)$ (or
conditional complexity $\KS(x|n)$; in fact, these requirements
are almost equivalent) is close to~$n$: the \emph{randomness
deficiency} $d(x)$ of $x$ is defined as $$d(x)=|x|-\KS(x)$$ (here $|x|$
stands for the length of~$x$). This sounds reasonable, but if we
then define an infinite random sequence as a sequence whose
prefixes have deficiencies bounded by a constant, such a
sequence does not exist at all: Martin-L\"of showed that every
infinite sequence has prefixes of arbitrarily large deficiency,
and suggested a different definition of randomness using
effectively null sets. Later more refined versions of randomness
deficiency (using monotone or prefix complexity) appeared that
make the criterion of randomness in terms of deficiencies
possible. But before that, in 1968, Kolmogorov wrote: 
\begin{quote}
The most
natural definition of infinite Bernoulli sequence is the
following: $x$ is considered $m$-Bernoulli type if $m$ is such
that all [its $i$-bit prefixes] $x^i$ are \emph{initial segments} of the finite
$m$-Bernoulli sequences. Martin-L\"of gives another, possibly
narrower definition~\cite[p.~663]{kolmogorov}.
\end{quote}

Here Kolmogorov speaks about ``$m$-Bernoulli'' finite sequence
$x$ (this means that $\KS(x|n,k)$ is greater than
$\log\binom{n}{k}-m$ where $n$ is the length of $x$ and $k$ is
the number of ones in $x$). We restrict ourselves to the case of
uniform Bernoulli measure where $p=q=1/2$. In this case
Kolmogorov's idea can be described as follows: an infinite
sequence is random if each its prefix also appears as a prefix
of some random string (=string with small randomness
deficiency). More formal, let us define
        $$
\bar d (x)=\inf \{d(y)\mid \text{$x$ is a prefix of $y$}\}
        $$
and require that $\bar d(x)$ is bounded for all prefixes of an
infinite sequence $\omega$. It is shown by Miller~\cite{miller}
that this definition is equivalent to Martin-L\"of randomness
relativized to $\mathbf{0}'$ (called also
$2$-\emph{randomness}):

\begin{theorem}[Miller]
        \label{kolmogorovrandomness}
A sequence $\omega$ is Martin-L\"of\, $\mathbf{0}'$-random if
and only if the quantities $\bar d(x)$ for all prefixes $x$ of
$\omega$ are bounded from above by a common constant.
\end{theorem}

There is another  related result proved in
Miller~\cite{miller} and Nies et al.~\cite{nies}:

\begin{theorem}[Miller, Nies, Stephan, Terwijn]
        \label{miller-nies}
A sequence $\omega$ is Martin-L\"of $\mathbf{0}'$-random if and
only if
        $$
\KS(\omega_0\omega_1\ldots\omega_{n-1})\ge n - c
        $$
for some $c$ and for infinitely many~$n$.
\end{theorem}

In the latter criterion the condition looks stronger: if $\KS(\omega_0\omega_1\ldots\omega_{n-1})\ge n - c$ for infinitely many $n$, then evidently $\bar d$ for all prefixes of $\omega$ is bounded by $c$. Theorem~\ref{miller-nies} can be reformulated as follows: the sequence $\omega$ is \emph{not} $\mathbf{0}'$-random if and only if $n-\KS(\omega_0\ldots\omega_{n-1})\to\infty$ as $n\to\infty$.

Let us show why theorems~\ref{kolmogorovrandomness} and~\ref{miller-nies} are consequences of Theorem~\ref{conidis}. In each direction we consider the stronger statement (among the two versions provided by theorems~\ref{kolmogorovrandomness} and~\ref{miller-nies}).

\begin{proof}
Assume that $n-\KS(\omega_0\ldots\omega_{n-1})\to\infty$ for some sequence $\omega$. We need to construct a $\mathbf{0}'$-effectively open set of small measure that contains $\omega$ (together with all other sequences with the same property).

Fix some $c$. For each $n$ consider the set $D_n^c$ of all
strings $u$ of length $n$ such that $\KS(u)<n-c$ (i.e.,
strings $u$ of length~$n$ such that $d(u)>c$). It has at most
$2^{n-c}$ elements. Then consider the set
        $$
U_n^c = \bigcup_{u \in D_n^c} [u]
        $$
(=~the set of all sequences that have prefixes in $D_n^c$). The set $U_n^c$ is effectively open uniformly in~$(n,c)$, since $D_n^c$ is enumerable uniformly in~$(n,c)$. Moreover, there are at most $2^{n-c}$ strings in $D_m^c$, hence the measure of $U_n^c$ is at most $2^{-c}$. The we can apply Theorem~\ref{conidis} to get an $\mathbf{0}'$-effectively open set of small measure (say, $2^{-(c-1)}$) that covers $\liminf_n U_n^c$. All the sequences that we need to cover belong to this $\liminf$ by definition. This proves the forward direction of the equivalence. (Remark: if we wanted to prove only the weaker statement from Theorem~\ref{kolmogorovrandomness}, the weaker version of Theorem~\ref{conidis}, with $\Int(V_N)$, would be enough.)

Consider now the reverse implication; we give the proof in terms of Martin-L\"of tests. (Miller~\cite{miller} provided a proof solely in terms of Kolmogorov complexity.)
Assume that a sequence $\omega$ is covered (for each $c$) by a
$\mathbf{0}'$-computable sequence of intervals $I_0,I_1,\ldots$
of total measure at most $2^{-c}$. (We omit $c$ in our notation,
but the construction below depends on~$c$.)

Using the approximations $\mathbf{0}_n$ of $\mathbf{0}'$
(obtained by performing at most $n$ steps of computation for
each $n$) we get another (now computable) family of intervals
$I_{0,n},I_{1,n},\ldots$ such that $I_{i,n}=I_i$ for every $i$
and sufficiently large $n$. We may assume without loss of
generality that $I_{i,n}$ either has size at least $2^{-n}$
(i.e., is determined by a string of length at most~$n$) or
equals $\bot$ (a special value that denotes the empty set) since
only the limit behavior is prescribed. Moreover, we may also
assume that $I_{i,n}=\bot$ for $n<i$ and that the total measure
of all $I_{0,n},I_{1,n},\ldots$ does not exceed $2^{-c}$ for
every $n$ (the latter is achieved by deleting the excessive
intervals in this sequence starting from the beginning; the
stabilization guarantees that all limit intervals will be
eventually let through).

Since $I_{i,n}$ is defined by intervals of size at least
$2^{-n}$, we get at most $2^{n-c}$ strings of length $n$ covered
by intervals $I_{i,n}$ for any given $n$ and all $i$. This set
of strings is decidable (recall that only $i$ not exceeding $n$
are used), therefore each string in this set can be determined,
assuming $c$ is known, by a string of length $n-c$, the binary
representation of its ordinal number in this set. Note that this
string also determines~$n$ if $c$ is known.

Returning to the sequence~$\omega$, we note that it is covered
by some $I_i$ and therefore is covered by $I_{i,n}$ for this $i$
and all sufficiently large $n$ (after the value of $I_{i,n}$ is
stabilized), say, for all $n\ge N$. Let $u$ be the prefix of
$\omega$ of length $N$. All extensions of $u$ of any length $n$
are covered by $I_{i,n}$ and thus have complexity less than
$n-c+O(1)$, conditional to~$c$, hence their complexity is at
most $n-c+2\log c+O(1)$. This means that $\bar d(u)\ge c - 2
\log c -O(1)$.

Such a string $u$ can be found for every $c$, therefore $\omega$
has prefixes of arbitrarily large $\bar d$-deficiency. This implies, in particular, that
$n-\KS(\omega_0\ldots\omega_{n-1})\to\infty$.
\end{proof}

%====================================================
\section{A generalization that is not possible}
%====================================================

The assumption of Theorem~\ref{conidis} was that \emph{all} $U_i$ have small measures: $\mu(U_i)\le \varepsilon$ for every $i$. In the classical measure-theoretic result one can replace this condition by a weaker one and require that infinitely many $U_i$ have small measure; it does not matter since we can delete all other $U_i$. Formally, one can note that
   $$
\mu(\liminf_i U_i)\le \liminf_i \mu(U_i).
   $$

As Conidis has shown, for the effective version of the statement the situation is different (and this is understandable, since we do not know which $U_i$ have small measure).

\begin{theorem}[Conidis]\label{conidis-negative}
  Theorem~\ref{conidis} is no more true if we require only that infinitely many $U_i$ have measure at most~$\varepsilon$.
\end{theorem}

\begin{proof}
    Recall Martin-L\"of's definition of randomness. The first level of an universal test is an effectively open set that covers all non-random reals (sequences) and has measure at most $1/2$. The complement of this set is an effectively closed set, and its minimal element is a lower semicomputable random number; we call it $\Omega$ (since it is closely related to Chaitin's Omega number).
    
    This consitruction can be relativized with oracle $\mathbf{0}'$: then we get a $\mathbf{0}'$-effectively open set of measure at most $1/2$ and $\Omega^{\mathbf{0}'}$, the minimal real outside it. This number is $\mathbf{0}'$-lower semicomputable, and it is easy to see that it can be represented as
    $$
\Omega^{\mathbf{0}'} = \liminf w_i,
    $$
where $w_i$ is a computable sequence. Now we show that for every rational $\varepsilon>0$ one can effectively construct a computable sequence of effectively open sets $U_{i,\varepsilon}$ such that 
    $$
\liminf_i \mu(U_{i,\varepsilon}) \le \varepsilon \quad \text{and}\quad \Omega^{\mathbf{0}'}\in \liminf_i U_i.    
    $$
If the strong version of Theorem~\ref{conidis} were true, we could conclude that $\Omega$ is not $\mathbf{0}'$-random, which is not the case.

It remains to construct the set $U_i$. One can let $U_i=(\inf_{j\ge i} w_j -\varepsilon/3, w_i+\varepsilon/3)$.
\end{proof}

\textbf{Remark}. This example shows only that an \emph{effective} transformation in Theorem~\ref{conidis} is not possible. However, Conidis (with a much more ingenious construction) has shown that there exists one specific computable sequence $U_i$ of effectively open sets such that $\liminf_i \mu(U_i)\le 1/2$  but $\liminf_i U_i$ cannot be covered by an $\mathbf{0}'$-effectively open set of a measure $3/4$. 

\section{Effective Fatou's lemma}

The results discussed above may be considered as constructive versions of classical
Fatou's lemma. This lemma says that if $\int f_i(x)\,d\mu(x)\le\varepsilon$ for $\mu$-measurable functions
$f_0,f_1,f_2,\ldots$, then
        $$
\int \liminf_{i\to+\infty} f_i (x)\,d\mu(x) \le \varepsilon.
        $$
Its constructive version can be formulated as follows:

\begin{theorem}
Let $f_i$ be a computable sequence of lower semicomputable functions such that $\int f_i(x)\, d\mu(x)$ does not exceed some rational $\varepsilon$ for all~$i$. Then for every $\varepsilon'>\varepsilon$ one can effectively construct a lower $\mathbf{0}'$-semicomputable function $\varphi$ such that 
     $$\liminf\,f_n (x)\le \varphi(x)\text{ \ for every $x$, \ and }\int\varphi(x) d\mu(x)\le\varepsilon'.$$
\end{theorem}
This is a natural generalization of the statement of Theorem~\ref{conidis} (which considers the special case when functions are indicator functions of open sets) and may be proved by essentially the same argument. 

To make the statement precise, we need to say on which space all $f_i$ are defined. We do not try to formulate this statement in full generality and note only that we can  consider Cantor space, the discrete space $\mathbb{N}$ or reals (and the same proof works).

\begin{proof}
For all integers $m$, for all positive rational numbers $r$, and  and for each open interval $U$ we consider an auxiliary function $u=r\chi_U$ (which is equal to $r$ inside $U$ and is equal to $0$ elsewhere), and try to increase all $f_m,f_{m+1},\ldots$ up to $u$:
  $$ f_s := \max (f_s,u)$$
(for $s=m,m+1,m+2,\ldots$) in the hope that the integral of $f_s$ still does not exceed the (increased) threshold. Note that the function $u$ is lower semicomputable, the maximum of two lower semicomputable functions is lower semicomputable, and therefore crossing the threshold is an enumerable event that can be checked using $\mathbf{0}'$. If we encounter some $s$ where this integral exceeds the threshold, we trim $u$:
     $$u:=\min (u,f_s)$$
and start over,  increasing all $f_m,f_{m+1},\ldots$ up to (new) $u$. Now we can make at least one step more, since the function that created troubles is now used as a cap and for this $s$ the integral does not exceed even the old threshold. But we may get again into troubles on some later stage $s'>s$. In this case we use $f_{s'}$ as the cap, too: 
   $$u:=\min(u,f_{s'}).$$
And so on. Note that the overflow can happen only finitely many times (for the same reasons as before: after each trimming the integral of $u$ decreases at most by $\varepsilon'-\varepsilon$, the increase of the threshold). So finally we get a lower semicomputable function whose integral does not exceed the increased threshold, and proceed to the next triple $(m,r,U)$.

The $\mathbf{0}'$-lower semicomputable function $\varphi$ that we need to construct can be defined now as the supremum of all the functions $u$ constructed on all steps. The integral of this function cannot be large, since for any finite set of $u$-functions the supremum of them (even together with one of $f_i$) was below the threshold.

If $\liminf_i f_i(x)>r$ for some $r$, then $f_i(x)>r$ for all $i$ starting from some $N$. Take some interval $U$ that contains $x$, and start adding $r\chi_U$ to $f_N, f_{N+1}, \ldots$. Since $f_i(x)>r$ for $i\ge N$, the trimming will not change the value of $u(x)$, so after this step the value at $x$ exceeds $r$.
\end{proof}

In this way we can also get the results of Sections~\ref{sec:plain-complexity} and~\ref{sec:prefix-complexity} as corollaries.

\vspace{2mm}

\textbf{Acknowledgments.} The authors are thankful to Steve   Simpson, Bjorn Kjos-Hanssen and Peter Cholak for useful discussions, to the members of LIF and Poncelet laboratories, to the participants of Kolmogorov seminar and to two anonymous referees of~\cite{stacs-tocs} for their numerous comments and suggestions. As we have said, the nice direct argument for Conidis' theorem (Theorem~\ref{conidis}) was suggested by Bruno Bauwens~\cite{bauwens}.

\end{document}